\documentclass{SCIYA2017enOL}
\online
\begin{document}

\ensubject{fdsfd}%二级学科

%%%%%%%%%%%%%%%%%%%%%%%%%%%%%%%%%%%%%%%%%%%%%%%%%%%%%%%
%%% Authors do not modify the information below
%%% 作者不需要修改此处信息
%%% 有专题名称时, 将第一行的{}注释掉, 使用第二行
\ArticleType{ARTICLES}%栏目
%\SpecialTopic{Progress of Projects Supported by NSFC}%专题
%\SubTitle{Dedicated to Professor Yang Lo on the Occasion of his {\rm 70}th Birthday}%专刊说明
\Year{2021}
\Month{November}%
\Vol{60}
\No{1}
\BeginPage{1} %
\DOI{....}
\ReceiveDate{June 22, 2021}
\AcceptDate{November 29, 2021}
%\OnlineDate{January 1, 2017}
%%%%%%%%%%%%%%%%%%%%%%%%%%%%%%%%%%%%%%%%%%%%%%%%%%%%%%%

%%% title: 标题
%%%   \title{title}{title for citation}
\title[]{Further study on Horozov-Iliev's method of estimating the number of limit cycles }
{The number of zeros of Melnikov function}%%后边花括号是眉题

%%% Corresponding author: 通信作者
%%%   \author[number]{Full name}{{email@xxx.com}}
%%% General author: 一般作者
%%%   \author[number]{Full name}{}
\author[1]{Xiaoyan Chen}{{xychen@hnu.edu.cn}}
\author[2,$\ast$]{Maoan Han}{{mahan@shnu.edu.cn}}
%\author[4]{Givenname Familyname4}{fourth@mail.com}

%%% Author information for page head. 页眉中的作者信息
%%% 若此处指定以此处为准, 否则直接调用author信息
\AuthorMark{Xiaoyan Chen}

%%% Authors for citation. 首页引用中的作者信息
%%% 若此处指定以此处为准, 否则直接调用author信息
\AuthorCitation{Author A, Author B, Author C}

%%% Address. 地址
%%%   \address[number]{Address, City {\rm Postcode}, Country}
\address[1]{Department of Computer Science and Mathematics, Changsha University, Changsha, {\rm410022}, China}
\address[2]{Department of Mathematics, Zhejiang Normal
University, Zhejiang, {\rm321004}, China}
%\address[4]{Wuhan Institute of Physics and Mathematics, Chinese Academy of Sciences, Wuhan {\rm 430071}, China}

%%% Abstract. 摘要
\abstract{In the study of the number of limit cycles of near-Hamiltonian systems, the first order Melnikov function plays an important role. This paper aims to establish a development of a known method to estimate the upper bound of the number of zeros of the function.}

%%% Keywords. 关键词

%%% MSC numbers. Requested items

\maketitle

%%%%%%%%%%%%%%%%%%%%%%%%%%%%%%%%%%%%%%%%%%%%%%%%%%%%%%%
%%% The main text. 正文部分%
%%  图表引用\cref公式引用\eqref参考文献\cite
%%%%%%%%%%%%%%%%%%%%%%%%%%%%%%%%%%%%%%%%%%%%%%%%%%%%%%%
\section{Introduction}

Consider the planar differential system
 \begin{equation}\label{PQ1}
 \left\{ \!\!\!
\begin{array}{ll}
 \dot{x}=P(x,y),\\
\dot{y}=Q(x,y),\\
\end{array}
\right.
 \end{equation}
where $P(x,y)$ and $Q(x,y)$ are real polynomials in the variables $x$ and $y$ of degree less or equal to $n$. As we know, the determination of the number and positions of limit cycles for system \eqref{PQ1} is the second part of Hilbert's 16th problem \cite{Hilbert1902}. This problem is difficult, and is still open even for quadratic systems. In 1983, Arnold posed weak Hilbert's 16th problem which asks the maximum number of zeros of the Abelian integral
$$ I(h)=\oint_{H=h} gdx-fdy,$$
 where $H$, $f$, and $g$ are all real polynomial functions of $x$ and $y$. This problem has been studied by many researchers (see \cite{Hanbook,Han2012,Li2000} and references there in ).

 In this paper, we consider a system of the form
  \begin{equation}\label{I1}
\left\{ \!\!\!
\begin{array}{ll}
\dot{x}=\frac{\partial H(x,y)}{\partial y}+\epsilon p(x,y),\\
\dot{y}=-\frac{\partial H(x,y)}{\partial x}+\epsilon q(x,y),\\
\end{array}
\right.
\end{equation}
where $\epsilon \in \mathbb{R}$ is a small parameter, $H(x,y)$, $p(x,y)$, and $q(x,y)$ are smooth functions of $x$ and $y$. Suppose that system $\eqref{I1}|_{\epsilon=0}$ has a family of periodic orbits $L_h$ defined by $H(x,y)=h$ for $h\in (\alpha, \beta)$.

 As we know, the first order Melnikov function of system \eqref{I1} has the following form
 $$
 M(h)=\int_{L_h}q(x,y)dx-p(x,y)dy,\quad h\in (\alpha, \beta),
$$
see \cite{Hanbook}.
In 2010, the authors in \cite{Liu2010} introduced the first order Melnikov function and established its formula when \eqref{I1} is a piecewise smooth system. Recently, the formula of the first order Melnikov function $M(h)$ was extended to piecewise smooth near-Hamiltonian or near-integrable systems with multiple switching lines or curves, see \cite{Tian2021,Wang2016}. As we know, the total number of zeros of $M(h)$ can control the number of limit cycles bifurcating from a period annulus. More precisely, Han and Yang in \cite{Han2021} showed that if $M(h)$ has at most $k$ zeros in $h\in(\alpha,\beta)$, multiplicity taken into account, then for small $|\epsilon|>0$, system \eqref{I1} has at most $k$ limit cycles bifurcating from the period annulus defined by $L_h$, multiplicity taken into account. One can find in \cite{Hanbook,Han2015,Han2021,Liu2010} more information on the relationship between the number of zeros of $M(h)$ and the number of limit cycles.

The main aim of this paper is to establish a development of a known method to estimate upper bound of the maximum number of zeros of $M(h)$ in $h$ on $(\alpha,\beta)$. Up to now, there have been certain classical and fundamental approaches on this aspect. Let us briefly introduce some of them as follows.

(i) Horozov-Iliev's method. Suppose that there exist functions $G$ and $F$ defined on $(\alpha,\beta)$ such that
\begin{equation}\label{HIM1}
 \left(\frac{M(h)}{G(h)}\right)'=\frac{F(h)}{G^2(h)},\quad h\in (\alpha,\beta)\setminus \mathbb{S},
 \end{equation}
where $\mathbb{S}$ is the set of the zeros of $G(h)$ on $(\alpha,\beta)$. Then, we have
\begin{equation}\label{HIE1}
\lambda\leq \mu+p+1,
\end{equation}
where $\lambda$, $\mu$, and $p$ are the number of zeros of $M(h)$, $F(h)$, and $G(h)$ on $(\alpha,\beta)$, respectively (see Lemma 4.2 in \cite{Horozov1998}). Obviously, \eqref{HIM1} provides a method to estimate the upper bound of the number of zeros of $M(h)$ on $(\alpha,\beta)$. Since it was firstly obtained by Horozov and Iliev in 1998, we call it Horozov-Iliev's method. In fact, this method has been widely applied by researchers in recent years, see \cite{Liang2012,Xiong2012,Yang2018,Zhao1999} and references therein.

(ii) Chebyshev criterion. Let $(f_1,f_2,\ldots, f_n)$ be an ordered set of $C^{\infty}$ functions on the interval $(\alpha,\beta)$. The ordered set $(f_1,f_2,\ldots,f_n)$ is called to be an extended complete Chebyshev system (in short ECT-system) on $(\alpha,\beta)$ if for all $1\leq k\leq n$, any nontrivial linear combination $\sum\limits_{i=1}^ka_if_i(h)=0$ has at most $k-1$ isolated zeros on $(\alpha,\beta)$ counted with multiplicities. Chebyshev criterion says that the ordered set $(f_1,f_2,\ldots,f_n)$ is an ECT-system if and only if for each $k=1,2,\ldots,n,$
$$ W_k(h)=\left |\begin{matrix}
 f_1(h) & f_2(h)& \cdots &  f_{k}(h) \\
f'_1(h)& f'_2(h)&\cdots &  f'_{k}(h)\\
\vdots&\vdots&\cdots&\vdots\\
f^{(k-1)}_1(h)& f^{(k-1)}_2(h)&\cdots &  f^{(k-1)}_{k}(h)\\
\end{matrix}\right |\neq 0, \quad h\in (\alpha,\beta),$$
see \cite{Karlin1966}. Chebyshev theory was also widely used and extended by many researchers, see \cite{Cen2018,Gasull2012,Manosas2011,Douglas2017,Yang20201,Zalik2011} and references therein.

(iii) Petrov's method or method of argument principle. This method was firstly used by G. S. Petrov to study the perturbations of elliptic Hamiltonian of degree $3$ and degree $4$ in 1980's (see \cite{Petrov1984,Petrov1986,Petrov1987,Petrov1988}). The main idea of this method is to extend $h$ to a complex number and estimates the number of zeros of $M(h)$ by argument principle in a suitable domain, see \cite{Iliev2003,Li2015,Yang2017} and references therein.

In \cite{Wang2016}, the authors considered a piecewise smooth near-Hamiltonian system of the following form
\begin{equation}\label{WHs}
\left (
\begin{array}{ll}
\dot{x}\\
\dot{y}\\
\end{array}
\right )= \left (
\begin{array}{ll}
\frac{\partial H_k(x,y)}{\partial y}+\epsilon p_k(x,y)\\
-\frac{\partial H_k(x,y)}{\partial x}+\epsilon q_k(x,y)\\
\end{array}
\right ),
\quad (x,y)\in D_k,\quad k=1,2,3,4,\\
\end{equation}
where
\begin{equation}\label{pqk}
p_k(x,y)=\sum\limits_{i+j=0}^{n}a_{kij}x^iy^j,\quad q_k(x,y)=\sum\limits_{i+j=0}^{n}b_{kij}x^iy^j ,
\end{equation}
$$D_1=\{(x,y)|x>0,y>0\}, \quad D_2=\{(x,y)|x<0,y>0\},$$
$$D_3=\{(x,y)|x<0,y<0\}, \quad D_4=\{(x,y)|x>0,y<0\},$$
$\epsilon $ is small enough, and $H_k\in C^{\infty}$. Specially, the authors in \cite{Wang2016} studied system \eqref{WHs} for the following four cases.
\begin{description}
\item Case $(1)$: $H_1(x,y)$ is given by
\begin{equation}\label{H1}
H_1(x,y)=\lambda_1(x-1)(y-1),\quad (x,y)\in D_1,
\end{equation}
and $H_k(x,y)$, $k=2,3,4$ are given by
\begin{equation}\label{Hk}
H_k(x,y)=-\frac{w_k}{2}(x^2+y^2),\quad (x,y)\in D_k,
\end{equation}
where $ \lambda_1>0$ and $w_k>0$.
\item Case $(2)$: $H_1(x,y)$ is given by \eqref{H1}, $H_k(x,y)$, $k=3,4$ are given by \eqref{Hk}, and
\begin{equation}\label{H2}
H_2(x,y)=-\lambda_2(x+1)(y-1),\quad (x,y)\in D_2,
\end{equation}
where $\lambda_2>0$.
\item Case $(3)$: $H_i(x,y)$, $i=1,2,4$ are given by \eqref{H1}, \eqref{H2}, \eqref{Hk}, respectively, and
\begin{equation}\label{H3}
H_3(x,y)=\lambda_3(x+1)(y+1),\quad (x,y)\in D_3,
\end{equation}
where $\lambda_3>0$.
\item Case $(4)$: $H_i(x,y)$, $i=1,2,3$ are given by \eqref{H1}, \eqref{H2}, \eqref{H3}, respectively, and
\begin{equation}\label{H4}
H_4(x,y)=-\lambda_4(x-1)(y+1),\quad (x,y)\in D_4,
\end{equation}
where $\lambda_4>0$.
\end{description}

The authors in \cite{Wang2016} obtained that for each one of the above cases system \eqref{WHs} has a family of periodic orbits given by $L_{h}=\cup_{k=1}^4L_{h}^k$ for $h\in (0,1)$,
where
$$L_{h}^k=\{(x,y)|H_k(x,y)=-\frac{w_k}{2}(1-h)^2,(x,y)\in D_k\}$$
for $H_k(x,y)$ defined by \eqref{Hk} and
$$L_{h}^k=\{(x,y)|H_k(x,y)=\lambda_k h,(x,y)\in D_k\}$$
for $H_k(x,y)$ defined by \eqref{H1}, \eqref{H2}-\eqref{H4}. Clearly,
 if $h\rightarrow 1 $, $L_h$ approaches the origin which is a center for cases $(1)$-$(3)$ and a generalized singular point for case $(4)$. If $h\rightarrow 0$,
 $L_h$ approaches a compound homoclinic loop $L_0$ with a saddle $(1,1)$ for case $(1)$ and a compound $k$-polycycle $L_0$ for case $(k)$, $k=2,3,4$. For the definition of a compound homoclinic loop or $k$-polycycle, see \cite{Wang2016}.

Denote by $Z_k(n)$ the maximum number of zeros of the first order Melnikov function of system \eqref{WHs} for case $(k)$ on the open interval $(0,1)$ for all possible $p_i$ and $q_i$ satisfying \eqref{pqk}, $i,k=1,2,3,4$. For case $(k)$, denote by $N_{Hopf}^k(n)$ the maximal number of limit cycles produced in Hopf bifurcation near the origin for all possible $p_i$ and $q_i$ satisfying \eqref{pqk}, $i=1,2,3,4$, $k=1,2,3$. Denote by $N_{Homoc}(n)$ and by $N_{Hetec}^k(n)$ respectively the maximal number of limit cycles produced in the compound homoclinic bifurcation and the compound $k$-polycycle bifurcation for all possible $p_i$ and $q_i$ satisfying \eqref{pqk} for case $(k)$, $i=1,2,3,4$, $k=2,3,4$. Then, the main results in \cite{Wang2016} can be stated as follows.

 \begin{theorem}[{See \cite{Wang2016}}]\label{71}
Consider system $(\ref{WHs})$ with $n\geq2$. We have
\begin{description}
\item $(i)$ $N_{Homoc}(n)\geq n+2+[\frac{n+1}{2}]$.
\item $(ii)$ $N^k_{Hopf}(n)\geq n+2+[\frac{n+1}{2}]$ for $k=1,2,3$.
\item $(iii)$ $ N^{k}_{Hetec} \geq n+2+[\frac{n+1}{2}]$ for $k=2,3$, and $ N^{4}_{Hetec} \geq n+1+[\frac{n+1}{2}]$,
\item $(iv)$ $ n+2+[\frac{n+1}{2}]\leq Z_k(n)\leq n+1+2[\frac{n+1}{2}]$ for $k=1,2,3$, and $ n+1+[\frac{n+1}{2}]\leq Z_4(n)\leq n+1+2[\frac{n+1}{2}]$.
\end{description}
\end{theorem}

In \cite{Yang2020}, the author considered the following two piecewise smooth near-integrable systems
 \begin{equation}\label{ruh2}
 \left (
\begin{array}{ll}
\dot{x}\\
\dot{y}\\
\end{array}
\right )= \left (
\begin{array}{ll}
\sqrt{2}xy+\epsilon f_k(x,y)\\
\frac{\sqrt{2}}{4}(1-x^2+2y^2)+\epsilon g_k(x,y)\\
\end{array}
\right ),
\quad (x,y)\in D_k,\quad k=1,2,3,4,\\
\end{equation}
and
\begin{equation}\label{yruh2}
 \left (
\begin{array}{ll}
\dot{x}\\
\dot{y}\\
\end{array}
\right )= \left (
\begin{array}{ll}
\frac{\sqrt{2}}{2}xy+\epsilon f_k(x,y)\\
\frac{\sqrt{2}}{2}(2-2x+y^2)+\epsilon g_k(x,y)\\
\end{array}
\right ),
\quad (x,y)\in D_k,\quad k=1,2,3,4,\\
\end{equation}
where
$$
f_k(x,y)=\sum\limits_{i+j=0}^{n}c_{kij}x^iy^j,\quad g_k(x,y)=\sum\limits_{i+j=0}^{n}d_{kij}x^iy^j ,
$$
$$D_1=\{(x,y)|x>1,y>0\},\quad D_2=\{(x,y)|x>1,y<0\} ,$$
 $$D_3=\{(x,y)|x<1, y<0\}, \quad D_4=\{(x,y)|x<1,y>0\},$$
and $\epsilon$ is small enough. From \cite{Yang2020}, system $\eqref{ruh2}|_{\epsilon=0}$ has two isochronous centers $(\pm1,0)$ and the families of periodic orbits around the centers $(\pm1,0)$ are given by $\overline{L}_h=\{(x,y)|H(x,y)=h\}$ for $h\in (-\infty,-1)\bigcup(0,+\infty) $, where
$$H(x,y)=\frac{1}{x}\left(\frac{1}{2}y^2+\frac{1}{4}x^2-\frac{1}{2}x+\frac{1}{4}\right).$$
One also knows from \cite{Yang2020} that system $\eqref{yruh2}|_{\epsilon=0}$ has one isochronous center $(1,0)$ and the family of periodic orbits around the center $(1,0)$ is given by $\widehat{L}_h=\{(x,y)|H(x,y)=h\}$ for $h\in (0,1)$, where
$$H(x,y)=\frac{1}{x^2}\left(\frac{1}{2}y^2+x^2-2x+1\right).$$
The main results obtained in \cite{Yang2020} are as follows.
\begin{theorem}[{See \cite{Yang2020}}]\label{301}

(i) The number of limit cycles of system \eqref{ruh2} bifurcating from the period annuli around the isochronous centers $(\pm1,0)$ is no more than $11n+34$ for $n\geq1.$

(ii) The number of limit cycles of system \eqref{yruh2} bifurcating from the period annulus around the isochronous center $(1,0)$ is no more than $21n-25$ for $n\geq 3$; $39$ for $n=1,2.$

\end{theorem}

This paper aims to establish a development of Horozov-Iliev's method and apply it to improve the two theorems above. The rest of this paper is organized as follows. In section 2, we establish and prove our main theorem which is  a development of Horozov-Iliev's method. In sections 3 and 4, we apply our main theorem to estimate the number of limit cycles of systems \eqref{WHs}, \eqref{ruh2} and \eqref{yruh2}, respectively. By using our main theorem, we obtain shaper and more precise estimation of the number of limit cycles for these three systems. In the last section, we present a conclusion.

\section{Development of Horozov-Iliev's method}\label{section2}
In this section, we establish the following theorem.
\begin{theorem}\label{upperbound1}
Let $ M:(\alpha,\beta)\longrightarrow \mathbb{R}$ be a $C^k$ function, $k\geq 1.$ Suppose that there exist two $C^k$ functions $G$ and $\widetilde{F}$ defined on $(\alpha,\beta)$ such that
\begin{equation}\label{GHI}
\left(\frac{M(h)}{G(h)}\right)^{(m)}=\frac{\widetilde{F}(h)}{G^{m+1}(h)},\quad h\in (\alpha,\beta)\setminus \mathbb{S},
\end{equation}
where $1\leq m\leq k$ and $\mathbb{S}$ is the set of zeros of $G(h)$ in $h$ on $(\alpha,\beta)$. Then, we have
\begin{equation}\label{te}
\lambda\leq \mu+mp+m,
\end{equation}
where $\lambda$, $\mu$, and $p$ are the number of zeros of $M(h)$, $\widetilde{F}(h)$, and $G(h)$ in $h$ on $(\alpha,\beta)$ including the multiplicities, respectively.

\end{theorem}

\begin{proof} We prove \eqref{te} for two cases of $p=0$ and $p\neq 0$, separately.

If $p=0$, then by \eqref{GHI}, $\left(\frac{M(h)}{G(h)}\right)^{(m)}$ has $\mu$ zeros on $(\alpha,\beta)$, multiplicity taken into account. According to Lemma 2.3 in \cite{Han2021}, it follows that $\left(\frac{M(h)}{G(h)}\right)^{(m-1)}$ has at most $\mu+1$ zeros on $(\alpha,\beta)$, multiplicity taken into account. In a similar way, we can obtain that $\frac{M(h)}{G(h)}$ has at most $\mu+m$ zeros on $(\alpha,\beta)$, multiplicity taken into account. Therefore, \eqref{te} holds for $p=0$.

If $p\neq 0$, there are exactly $s$ different zeros of $G(h)$ on $(\alpha,\beta)$, denoted by $h_{1}$, $\ldots,$ $h_{s}$, for some integer $s$ satisfying $1\leq s\leq p.$ We
 can assume that $h_{1}<h_{2}<\cdots<h_{s}$ and that they have multiplicities $a_1$, $\ldots$, $a_s,$ respectively. Then,
 $$a_1+\cdots+a_s=p, \quad a_1\geq 1,\ldots, a_s\geq 1.$$
 For $i=1,2,\ldots,s$, we assume that $h_{i}$ is a zero of $M(h)$ and $\widetilde{F}(h)$ with multiplicities $l_i$ and $t_i$, respectively. If $M(h_{i})\neq0$ ($\widetilde{F}(h_{i})\neq0 $) for some $i$, then we take $l_i=0$ $(t_i=0)$. Thus, $l_i\geq0$ and $t_i\geq0$ for $i=1,2,\ldots,s.$ Next, we prove $l_i\leq t_i$, $i=1,2,\ldots, s.$ By Lemma 2.2 in \cite{Han2021}, there exist functions $m_i\in C^{k-l_i}$, $g_i\in C^{k-a_i}$, $ f_i\in C^{k-t_i}$ satisfying $m_i(h_{i})\neq 0$, $g_i(h_{i})\neq 0$, $f_i(h_{i})\neq 0$ such that for $i=1,2,\ldots,s$,
  $$M(h)=(h-h_{i})^{l_i}m_i(h),$$
  $$\quad G(h)=(h-h_{i})^{a_i}g_i(h),$$
  $$ \widetilde{F}(h)=(h-h_{i})^{t_i}f_i(h). $$
Then, for $i=1,2,\cdots,s$,
\begin{equation}\label{MG}
\frac{M(h)}{G(h)}=\frac{(h-h_{i})^{l_i-a_i}m_i(h)}{g_i(h)}
 \end{equation}
and
\begin{equation}\label{MG2}
\frac{\widetilde{F}(h)}{G^{m+1}(h)}=\frac{(h-h_{i})^{t_i-(m+1)a_i}f_i(h)}{[g_i(h)]^{m+1}}.
\end{equation}
Thus, by \eqref{MG}, for $i=1,2,\ldots,s$,
$$ \left(\frac{M(h)}{G(h)}\right)^{(m)}=(h-h_{i})^{l_i-a_i-m}\phi_i(h), \quad 0<|h-h_{i}|\ll1,$$
where the function $\phi_{i}$ is $C^{k-m}$ for $|h-h_{i}|\ll1$. Here, $\phi_{i}(h_i)$ may be zero or not zero, $i=1,2,\ldots,s$. Then, from \eqref{GHI} and \eqref{MG2}, we have
$$ l_i-a_i-m\leq t_i-(m+1)a_i,\quad i=1,2,\ldots,s,$$
i.e.,
$$l_i\leq t_i-ma_i+m,\quad i=1,2,\ldots,s.$$
 Hence, by $a_i\geq 1$, we have $l_i\leq t_i$, $i=1,2,\ldots,s.$

Let $h_{0} =\alpha$ and $h_{s+1}=\beta$, and we assume that $M(h)$ and $\widetilde{F}(h)$ respectively have $\lambda_i$ and $\mu_i$ zeros on $(h_{i},h_{i+1})$, multiplicity taken into account, $i=0,1,\ldots, s$. Then, by the proof of the case $p=0$, we have $\lambda_i\leq \mu_i+m$, $i=0,1,\ldots, s$.
Therefore, the total number of zeros of $M(h)$ on $(\alpha,\beta)$ is estimated by
$$\aligned
\lambda&=\sum\limits_{j=1}^s l_j+\sum\limits_{j=0}^s \lambda_j\\
&\leq \sum\limits_{j=1}^s t_j +\sum\limits_{j=0}^s (\mu_j+m)\\
&= \sum\limits_{j=1}^s t_j+\sum\limits_{j=0}^s \mu_j+(s+1)m\\
&=\mu+(s+1)m\\
&\leq \mu+mp+m,
\endaligned$$
since $s\leq p$. The proof is ended.
\end{proof}
\begin{remark}
From the above proof, it is easy to see that Theorem \ref{upperbound1} also holds for $M(h)$, $G(h)$, and $\widetilde{F}(h)$ defined in closed interval or half-closed and half-open interval or unbounded interval. Obviously, the method to find an upper bound of the number of zeros of $M(h)$ introduced in Theorem \ref{upperbound1} is a generalization of Horozov-Iliev's method. Moreover, it improves Horozov-Iliev's method, due to multiplicity taken into account in Theorem \ref{upperbound1}.
\end{remark}

\section{The number of limit cycles of system $(\ref{WHs})$  }\label{section3}
 In this section, we estimate the number of limit cycles bifurcating from the period annulus $\cup_{
   h\in (0,1) }L_{h}$ of system \eqref{WHs} for cases $(1)$-$(4)$. For this purpose, we first introduce some conclusions obtained in \cite{Wang2016}.

For $n\geq 2$, denote by $M_k(h)$ the first order Melnikov function of system \eqref{WHs} for case $(k)$, $k=1,2,3,4$. Then, from Lemmas 4.1, 5.1, 6.1, and 7.1 in \cite{Wang2016}, we have for $k=1,2,3$,
\begin{equation}\label{systemwM1}
\aligned
M_k(h)&=\sum\limits_{i=0}^{n+1}u_{ki}(1-h)^{i+1}+\sum\limits_{i=1}^nv_{ki}h^i(1-h)+
\sum\limits_{i=1}^{[\frac{n+1}{2}]}r_{ki}h^i(1-h)\ln h,\quad h\in(0,1),\\
\endaligned
\end{equation}
and
\begin{equation}\label{systemwM}
\aligned
M_4(h)&=\sum\limits_{i=0}^{n}u_{4i}(1-h)^{i+1}+\sum\limits_{i=1}^nv_{4i}h^i+
\sum\limits_{i=1}^{[\frac{n+1}{2}]}r_{4i}h^i\ln h,\quad h\in(0,1),\\
\endaligned
\end{equation}
where $u_{ki}$, $v_{kj}$, $r_{km}$, $i=0,1,\cdots,n$, $j=1,\cdots,n$, $m=1,2,\cdots,[\frac{n+1}{2}]$, $k=1,2,3,4$, and $u_{s,n+1}$, $s=1,2,3$ are independent coefficients and can be taken as free parameters.

Now, we are in a position to show one of our main results.
\begin{theorem}\label{GHIR1}
Consider system \eqref{WHs} with $n\geq2$. We have

\begin{description}
\item $(i)$ $Z_k(n)=n+2+[\frac{n+1}{2}]$ for $k=1,2,3$, and $Z_4(n)=n+1+[\frac{n+1}{2}]$.

\item $(ii)$ The number of limit cycles bifurcating from the period annulus $\cup_{
   h\in (0,1) }L_{h}$ is $n+2+[\frac{n+1}{2}] $ for case $(k)$, $k=1,2,3$ and $n+1+[\frac{n+1}{2}] $ for case $(4)$, multiplicity taken into account, by the first order Melnikov function.
\end{description}
\end{theorem}
\begin{proof} Obviously, $M_k(h)$ given in \eqref{systemwM1} can be analytically extended to the interval $(0,1]$, $k=1,2,3$. Then, for $k=1,2,3$, $M_k(h)$ can be rewritten as
$$M_k(h)=P_{k,n+2}(h)+\sum\limits_{i=1}^{[\frac{n+1}{2}]+1}A_{ki}h^i\ln h,\quad 0<h\leq1,$$
where $P_{k,n+2}(h) $ is a polynomial of degree $n+2$, and
$$ A_{ki}=\left\{ \!\!\!
\begin{array}{lll}
r_{k1}, \quad &i=1, \\
r_{ki}-r_{k, i-1},\quad &1<i<[\frac{n+1}{2}]+1,\\
-r_{k[\frac{n+1}{2}]},\quad & i=[\frac{n+1}{2}]+1.
\end{array}
\right.$$
Consider the $\left([\frac{n+1}{2}]+2\right)$th order derivative of $M_k(h)$, $k=1,2,3$. We have
$$\left(M_k(h)\right)^{([\frac{n+1}{2}]+2)}=P_{k,n-[\frac{n+1}{2}]}(h)+\sum\limits_{i=1}^{[\frac{n+1}{2}]+1}A_{ki}\left(h^i\ln h\right)^{([\frac{n+1}{2}]+2)}, \quad k=1,2,3.$$
Notice that
$$\left(\ln h\right)^{(m)}=\frac{(-1)^{m-1}(m-1)!}{ h^{m}}, \quad m\geq 1.$$
Then, for $m\geq i+1$, we have
$$\left(h^i\ln h\right)^{(m)}=\sum\limits_{j=0}^iC_{m}^ji(i-1)\cdots(i-j+1)h^{i-j}\frac{(-1)^{m-j-1}(m-j-1)!}{h^{m-j}}=\frac{B_{mi}}{h^{m-i}},
$$
where $B_{mi}=\sum\limits_{j=0}^iC_{m}^ji(i-1)\cdots(i-j+1)(-1)^{m-j-1}(m-j-1)!$. Hence, we have
$$\aligned
\left(M_k(h)\right)^{([\frac{n+1}{2}]+2)}&=P_{k,n-[\frac{n+1}{2}]}(h)+\sum\limits_{i=1}^{[\frac{n+1}{2}]+1}\frac{A_{ki}B_{[\frac{n+1}{2}]+2,i}}{h^{[\frac{n+1}{2}]+2-i}}\\
&=P_{k,n-[\frac{n+1}{2}]}(h)+\frac{\sum\limits_{i=1}^{[\frac{n+1}{2}]+1} A_{ki}B_{[\frac{n+1}{2}]+2,i}h^i}{h^{[\frac{n+1}{2}]+2}}\\
&=P_{k,n-[\frac{n+1}{2}]}(h)+\frac{\sum\limits_{i=0}^{[\frac{n+1}{2}]} A_{k,i+1}B_{[\frac{n+1}{2}]+2,i+1} h^i}{h^{[\frac{n+1}{2}]+1}}\\
&=\frac{P_{k,n+1}(h)}{h^{[\frac{n+1}{2}]+1}},\quad k=1,2,3.\\
\endaligned
$$
Then, by Theorem \ref{upperbound1}, we have that for each $k=1,2,3$, $M_k(h) $ has at most $ n+[\frac{n+1}{2}]+3$ zeros in $(0,1]$, multiplicity taken into account. From \eqref{systemwM1}, we have $M_k(1)=0$ for $k=1,2,3$. Thus, for each $k=1,2,3$, $M_k(h)$ has at most $ n+[\frac{n+1}{2}]+2$ zeros in $(0,1)$, multiplicity taken into account. Combining with Theorem \ref{71}, we have that $Z_k(n)=n+[\frac{n+1}{2}]+2$ for each $k=1,2,3$ and $n\geq 2$.

Similarly, $M_4(h)$ can be rewritten as
$$M_4(h)=P_{n+1}(h)+\sum\limits_{i=1}^{[\frac{n+1}{2}]}r_{4i}h^i\ln h,\quad 0<h<1.$$
Consider the $\left([\frac{n+1}{2}]+1\right)$th order derivative of $M_4(h)$ as follows
$$ \aligned
\left(M_4(h)\right)^{([\frac{n+1}{2}]+1)}&=P_{n-[\frac{n+1}{2}]}(h)+\sum\limits_{i=1}^{[\frac{n+1}{2}]}\frac{r_{4i}B_{[\frac{n+1}{2}]+1,i}}{h^{[\frac{n+1}{2}]+1-i}}\\
&=P_{n-[\frac{n+1}{2}]}(h)+\frac{\sum\limits_{i=1}^{[\frac{n+1}{2}]} r_{4i}B_{[\frac{n+1}{2}]+1,i}h^i}{h^{[\frac{n+1}{2}]+1}}\\
&=P_{n-[\frac{n+1}{2}]}(h)+\frac{\sum\limits_{i=0}^{[\frac{n+1}{2}]-1} r_{4,i+1}B_{[\frac{n+1}{2}]+1,i+1} h^i}{h^{[\frac{n+1}{2}]}}\\
&=\frac{P_{n}(h)}{h^{[\frac{n+1}{2}]+1}}.\\
\endaligned$$
Then, by Theorem \ref{upperbound1}, $M_4(h)$ has at most $n+1+[\frac{n+1}{2}]$ zeros in $h$ on $(0,1)$, multiplicity taken into account. Combining with Theorem \ref{71}, we have $Z_4(n)=n+1+[\frac{n+1}{2}]$ for $n\geq 2$.

Notice that for $n\geq 2$ and $k=1,2,3$, $Z_k(n)=n+2+[\frac{n+1}{2}]$. Then, by Theorem 4.4 in \cite{Han2021}, we have that the number of limit cycles of system $(\ref{WHs})$ for cases $(1)$-$(3)$ bifurcating from the period annulus $\cup_{h\in (0,1) }L_{h}$ is at most $n+2+[\frac{n+1}{2}] $, multiplicity taken into account. Recall that Theorem \ref{71} states that the number of limit cycles produced in Hopf bifurcation near the center (i.e., $0<1-h\ll1$) is at least $n+2+[\frac{n+1}{2}]$. Hence, the number of limit cycles bifurcating from the period annulus $\cup_{h\in (0,1) }L_{h}$ is $n+2+[\frac{n+1}{2}] $ for cases $(1)$-$(3)$, multiplicity taken into account.

Similarly, we can obtain that the number of limit cycles bifurcating from the period annulus $\cup_{h\in (0,1) }L_{h}$ is $n+1+[\frac{n+1}{2}] $ for case $(4)$, multiplicity taken into account. The proof is finished.
\end{proof}
\begin{remark}
 Theorem \ref{GHIR1} is an improvement of Theorem $\ref{71}$ on the upper bound of $Z_k(n)$ for $k=1,2,3,4$. Moreover, Theorem \ref{GHIR1} provides the maximum number of limit cycles of system \eqref{WHs} bifurcating from the period annulus $\cup_{h\in (0,1) }L_{h}$ for cases $(1)$-$(4)$, multiplicity taken into account.

\end{remark}

\section{The number of limit cycles of systems \eqref{ruh2} and \eqref{yruh2}}\label{section4}
In this section, we estimate the number of limit cycles bifurcating from the period orbits around the isochronous centers for systems \eqref{ruh2} and \eqref{yruh2}, respectively. For this purpose, we introduce some results obtained in \cite{Yang2020}. In fact, the author in \cite{Yang2020} derived the first order Melnikov functions of systems \eqref{ruh2} and \eqref{yruh2} by establishing some Picard-Fuchs equations. Before presenting these first order Melnikov functions, some notations are given firstly.

For $h\in(0,+\infty)$, let
\begin{equation}\label{f4i}
\aligned
f_{i,1}(h)&=a_i(h^2+h)-\sqrt{2}b_i|h|^{\frac{3}{2}}-\sqrt{2}b_i(h^2+h)\arctan\sqrt{|h|},\\
f_{i,2}(h)&=c_i\sqrt{h^2+h}-2b_ih,\\
f_{i,3}(h)&=a_ih-\sqrt{2}b_i[(h+1)\arctan \sqrt{|h|}-\sqrt{|h|}],\\
f_{i,4}(h)&=\frac{c_i}{2}\ln|2\sqrt{h^2+h}+2h+1|, \\
\endaligned
\end{equation}
where $b_1=1$, $b_2=-1$, and $a_i$, $c_i$, $i=1,2$ are constants.

For $h\in(-\infty,-1)$, let
 \begin{equation}\label{u4i}
\aligned
u_{1}(h)&=-4\sqrt{h^2+h},\\
u_{2}(h)&=c_3 h, \\
u_{3}(h)&=c_3(h^2+h),\\
u_{4}(h)&=4\sqrt{h^2+h}-2(2h+1)\ln|2\sqrt{h^2+h}+2h+1|,
\endaligned
\end{equation}
where $c_3$ is a constant.

For $h\in (0,1)$, let
 \begin{equation}\label{g4i}
\aligned
g_{i,1}(h)&=\widetilde{a_i}h,\\
g_{i,2}(h)&=\widetilde{b_i}\sqrt{h}, \\
g_{i,3}(h)&=2\widetilde{a_i}-2\widetilde{a_i}\sqrt{1-h}+\widetilde{c_i}\sqrt{2h}-\sqrt{2}\widetilde{c_i}\sqrt{1-h}\arcsin\sqrt{h},\\
g_{i,4}(h)&=2\widetilde{b_i}\sqrt{h}-\widetilde{b_i}(1-h)\ln\frac{1+\sqrt{h}}{1-\sqrt{h}}+\widetilde{c_i}(1-h)\ln(1-h)+\widetilde{c_i}h,
\endaligned
\end{equation}
where $ \widetilde{c_1}=1$, $ \widetilde{c_2}=-1$, and $\widetilde{a_i} $, $\widetilde{b_i}$, $i=1,2$ are constants.

The following lemma gives the first order Melnikov functions of systems \eqref{ruh2} and \eqref{yruh2}.
\begin{lemma}[{See \cite{Yang2020}}] (i) Consider system \eqref{ruh2}. For $h\in(0,+\infty)$,
\begin{equation}\label{ruh2M1}
M(h)=\sum\limits_{i=1}^2[\alpha_{i}(h)f_{i,1}(h)+\beta_{i}(h)f_{i,2}(h)+\gamma_i(h)f_{i,3}+\delta_i(h)f_{i,4}(h)]+P_{2n-1}(\sqrt{h}),
\end{equation}
where $P_{2n-1}(u)$ is a polynomial in $u$ of degree at most $2n-1$, and $\alpha_{k}(h)$, $\beta_k(h)$, $\gamma_k(h)$, $\delta_k(h)$ are polynomials of $h$ with
$$ {\rm{deg}}\alpha_k(h), {\rm{deg}}\beta_k(h)\leq n-2, \quad {\rm{deg}}\gamma_k(h)\leq n-1,\quad {\rm{deg}}\delta_k(h)\leq 2, \quad n\geq 3,$$
$$\alpha_k(h)\equiv constant,\quad {\rm{deg}}\beta_k(h), {\rm{deg}}\gamma_k(h), {\rm{deg}}\delta_k(h)\leq 1, \quad n=1,2,\quad k=1,2. $$
For $h\in(-\infty, -1)$,
\begin{equation}\label{ruh2M2}
 M_0(h)=\frac{1}{2h+1}[\alpha_{3}(h)u_{1}(h)+\beta_{3}(h)u_{2}(h)+\gamma_3(h)u_{3}(h)+\delta_3(h)u_{4}(h))],
 \end{equation}
where $\alpha_3(h)$, $\beta_3(h)$, $\gamma_3(h)$, and $\delta_3(h)$ are polynomials of $h$ with
$$ {\rm{deg}}\alpha_3(h), {\rm{deg}}\beta_3(h)\leq n-1, \quad {\rm{deg}}\gamma_3(h)\leq n-3,\quad {\rm{deg}}\delta_3(h)\leq 2, \quad n\geq 3,$$
$${\rm{deg}}\alpha_3(h)\leq 2,\quad  {\rm{deg}}\beta_3(h), {\rm{deg}}\gamma_3(h), {\rm{deg}}\delta_3(h)\leq 1,\quad  n=1,2. $$

(ii) Consider system \eqref{yruh2}. For $h\in(0,1)$, if $n\geq 3$, then
\begin{equation}\label{yruh2M1}
M(h)=\frac{1}{(h-1)^{n-2}}\{\sum\limits_{i=1}^2[\alpha_{i}(h)g_{i,1}(h)+\beta_{i}(h)g_{i,2}(h)+\gamma_i(h)g_{i,3}(h)+\delta_i(h)g_{i,4}(h)]+P_{3n-3}(\sqrt{h})\},
\end{equation}
where $P_{3n-3}(u) $ is a polynomial in $u$ of degree at most $3n-3$, and $\alpha_k(h)$, $\beta_k(h)$, $\gamma_k(h)$, $\delta_k(h)$ are polynomials of $h$ with
$$ {\rm{deg}}\alpha_k(h), {\rm{deg}}\delta_k(h)\leq n-2, \quad {\rm{deg}}\beta_k(h), {\rm{deg}}\gamma_k(h)\leq n-1, \quad k=1,2.$$
If $n=1,2$, then
\begin{equation}\label{yruh2M2}
M(h)=\frac{1}{h-1}\{\sum\limits_{i=1}^2[\alpha_{i}(h)g_{i,1}(h)+\beta_{i}(h)g_{i,2}(h)+\gamma_i(h)g_{i,3}(h)+\delta_i(h)g_{i,4}(h)]+P_{5}(\sqrt{h})\},
\end{equation}
where $P_{5}(u) $ is a polynomial in $u$ of degree at most $5$, and $\alpha_k(h)$, $\beta_k(h)$, $\gamma_k(h)$, $\delta_k(h)$ are polynomials of $h$ with
$$ {\rm{deg}}\alpha_k(h), {\rm{deg}}\delta_k(h)\leq 1, \quad {\rm{deg}}\beta_k(h), {\rm{deg}}\gamma_k(h)\leq 2, \quad k=1,2.$$

\end{lemma}

Now, we estimate the upper bound of the number of limit cycles for systems \eqref{ruh2} and \eqref{yruh2} by Theorem \ref{upperbound1}. Our results are shown in the following two theorems.
\begin{theorem}\label{yfirst}
For system \eqref{ruh2}, the number of limit cycles bifurcating from the period annulus $ \cup_{h\in (-\infty,-1) \bigcup(0, +\infty) }\overline{L}_h$ around the isochronous centers $(\pm1,0)$ is no more than $8n+2$ for $n\geq 3$, counting multiplicity; $21$ for $n=1,2$, counting multiplicity, by the first order Melnikov function.
\end{theorem}
\begin{proof} By \eqref{f4i}, \eqref{u4i}, \eqref{ruh2M1}, and \eqref{ruh2M2}, we have for $h\in (0,+\infty)$,
$$
M(h)=
\left\{\!\!\!
\begin{array}{ll}
P_2(h)+P_{1}(h)\sqrt{h}+\overline{P}_{1}(h)\sqrt{h^2+h}+P_2(h)\arctan\sqrt{h}+\widetilde{P}_1(h)g(h),\quad n=1,2,\\ \\
\overline{P}_n(h)+P_{n-1}(h)\sqrt{h}+P_{n-2}(h)\sqrt{h^2+h}+P_n(h)\arctan\sqrt{h}+P_2(h)g(h),\quad n\geq 3,\\
\end{array}
\right.
$$
 and for $h\in(-\infty,-1)$,
 $$M_0(h)=\frac{1}{2h+1}\left\{\!\!\!
\begin{array}{ll}
P_3(h)+P_{2}(h)\sqrt{h^2+h}+\overline{P}_2(h)g(h),\quad &n=1,2,\\ \\
P_n(h)+P_{n-1}(h)\sqrt{h^2+h}+P_3(h)g(h),\quad &n\geq 3,\\
\end{array}
\right.
$$
where $ g(h)=\ln |2\sqrt{h^2+h}+2h+1|$, and $P_i(h)$, $\overline{P}_i(h)$, $ \widetilde{P}_i(h)$ are polynomials of degree no more than $i$.
  For $n\geq 3$, consider the $(n+1)$th order derivative of $M(h)$. For this purpose, we first notice that
$$
\left(\arctan \sqrt{h}\right)'=\frac{1}{2(1+h)\sqrt{h}}\quad {\rm{and}} \quad \left(\ln(2\sqrt{h^2+h}+2h+1) \right)'=\frac{1}{\sqrt{h}\sqrt{1+h}}.
$$
Then, according to
$$ \left(\frac{1}{\sqrt{h}}\right)^{(n)}=\frac{(-1)^n(2n-1)!!}{2^nh^{n+\frac{1}{2}}},$$
$$ \left(\frac{1}{\sqrt{1+h}}\right)^{(n)}=\frac{(-1)^n(2n-1)!!}{2^n(1+h)^{n+\frac{1}{2}}}, $$
and
$$\left(\frac{1}{1+h}\right)^{(n)}=\frac{(-1)^nn!}{(1+h)^{n+1}},$$
we have
$$\left(\arctan \sqrt{h}\right)^{(n+1)}=\frac{\widehat{P}_n(h)}{(1+h)^{n+1}h^{n+\frac{1}{2}}}$$
and
$$ \left(\ln(2\sqrt{h^2+h}+2h+1) \right)^{(n+1)}=\frac{\overline{P}_n(h)}{h^{n+\frac{1}{2}}(1+h)^{n+\frac{1}{2}}}, $$
where $\widehat{P}_n(h) $ and $ \overline{P}_n(h)$ are polynomials of degree no more than $n$.
Similarly, according to
$$
\left(\sqrt{h}\right)^{(n)}=\frac{(-1)^{n-1}(2n-3)!!}{2^nh^{n-\frac{1}{2}}}\quad {\rm{and}}\quad \left(\sqrt{1+h}\right)^{(n)}=\frac{(-1)^{n-1}(2n-3)!!}{2^n(1+h)^{n-\frac{1}{2}}},
$$
we obtain
$$ \left(\sqrt{h(1+h)}\right)^{(n)}=\frac{\widetilde{P}_n(h)}{h^{n-\frac{1}{2}}(1+h)^{n-\frac{1}{2}}},$$
where $ \widetilde{P}_n(h)$ is a polynomial of degree no more than $n$.
Based on the above results, we have for $n\geq 3$,
$$\aligned
\left(M(h)\right)^{(n+1)}&=\sum\limits_{i=0}^{n-1}C_{n+1}^iP_{n-1-i}(h)\left(\sqrt{h}\right)^{(n+1-i)}+\sum\limits_{i=0}^{n-2}C_{n+1}^iP_{n-2-i}(h)\left(\sqrt{h^2+h}\right)^{(n+1-i)}\\
&+\sum\limits_{i=0}^{n}C_{n+1}^iP_{n-i}(h)\left(\arctan \sqrt{h}\right)^{(n+1-i)}+\sum\limits_{i=0}^{2}C_{n+1}^iP_{2-i}(h)\left(g(h)\right)^{(n+1-i)}\\
&=\sum\limits_{i=0}^{n-1}C_{n+1}^iP_{n-1-i}(h)\frac{(-1)^{n-i}(2n-2i-1)!!}{2^{n+1-i}h^{n-i+\frac{1}{2}}}+\sum\limits_{i=0}^{n-2}C_{n+1}^iP_{n-2-i}(h)\frac{\widetilde{P}_{n+1-i}(h)}{(h^2+h)^{n-i+\frac{1}{2}}}\\
&+\sum\limits_{i=0}^{n}C_{n+1}^iP_{n-i}(h)\frac{\widehat{P}_{n-i}(h)}{(1+h)^{n+1-i}h^{n-i+\frac{1}{2}}}+\sum\limits_{i=0}^{2}C_{n+1}^iP_{2-i}(h)\frac{\overline{P}_{n-i}(h)}{(h^2+h)^{n-i+\frac{1}{2}}}\\
&=\frac{P_{n-1}(h)}{h^{n+\frac{1}{2}}}+\frac{P_{2n-1}(h)}{(h^2+h)^{n+\frac{1}{2}}}+\frac{P_{2n}(h)}{(1+h)^{n+1}h^{n+\frac{1}{2}}}+\frac{P_{n+2}(h)}{(h^2+h)^{n+\frac{1}{2}}}\\
&=\frac{(1+h)^{n+1}P_{n-1}(h)+\sqrt{1+h}P_{2n-1}(h)+P_{2n}(h)+\sqrt{1+h}P_{n+2}(h)}{(1+h)^{n+1}h^{n+\frac{1}{2}}}\\
&=\frac{\overline{P}_{2n}(h)+\sqrt{1+h}P_{2n-1}(h)}{(1+h)^{n+1}h^{n+\frac{1}{2}} }.\\
\endaligned
$$
Then, $\left(M(h)\right)^{(n+1)}=0  $ is equivalent to
$$\frac{\overline{P}_{2n}(h)+\sqrt{1+h}P_{2n-1}(h)}{(1+h)^{n+1}h^{n+\frac{1}{2}} }=0. $$
It yields that $\left(M(h)\right)^{(n+1)} $ has at most $4n$ zeros. Then, according to Theorem \ref{upperbound1}, $M(h)$ has at most $ 5n+1$ zeros on the interval $(0,+\infty)$ for $n\geq 3$, multiplicity taken into account. Similarly, we have for $n=1,2$,
$$ \left(M(h)\right)^{(3)}=\frac{P_4(h)+\overline{P}_3(h)\sqrt{1+h}}{h^{\frac{5}{2}}(1+h)^{3}}.$$
Thus, by Theorem \ref{upperbound1}, $M(h)$ has at most $ 11$ zeros on the interval $(0,+\infty)$ for $n=1,2$, multiplicity taken into account.

For $h\in (-\infty,-1)$, we have
$$\left(\ln(-2\sqrt{h^2+h}-2h-1)\right)^{(n+1)}=\frac{P_n(h)}{(h^2+h)^{n+\frac{1}{2}} }.$$
So, if $n\geq3$, then the $(n+1)$th order derivative of $M_0(h)(2h+1)$ is
$$ \left(M_0(h)(2h+1)\right)^{(n+1)}=\frac{P_{2n}(h)}{(h^2+h)^{n+\frac{1}{2}}}.$$
 Then, by Theorem \ref{upperbound1}, $ M_0(h)$ has at most $3n+1$ zeros on $(-\infty,-1)$ for $n\geq 3$, multiplicity taken into account. Similarly, we have for $n=1,2$,
$$ \left(M_0(h)(2h+1)\right)^{(4)}=\frac{P_{6}(h)}{(h^2+h)^{\frac{7}{2}}}.$$
Then, by Theorem \ref{upperbound1}, $ M_0(h)$ has at most $10$ zeros on $(-\infty,-1)$ for $n=1,2$, multiplicity taken into account.

From the above analysis and Theorem 4.4 in \cite{Han2021}, we can conclude that the number of limit cycles of system \eqref{ruh2} bifurcating from the period annulus $ \cup_{h\in(-\infty,-1)\cup (0, +\infty)}\overline{L}_h$ around the isochronous centers $(\pm1,0)$ is no more than $8n+2$ (counting multiplicity) for $n\geq3$; $21$ (counting multiplicity) for $n=1,2.$ The proof is ended.
\end{proof}

\begin{theorem}\label{ysecond}
For system \eqref{yruh2}, the number of limit cycles bifurcating from the period annulus $ \cup_{h\in(0,1)}\widehat{L}_h$ around the isochronous center $(1,0)$ is no more than $15n-11$ (counting multiplicity) for $n\geq 3$; 28 (counting multiplicity) for $n=1,2,$ by the first order Melnikov function.
\end{theorem}
\begin{proof} According to \eqref{g4i} and \eqref{yruh2M1}, we have for $n\geq 3$,
$$\aligned
M(h)=&\frac{1}{(h-1)^{n-2}}[P_{n-1}(h)+\widetilde{P}_{n-1}(h)\sqrt{h}+\widehat{P}_{n-1}(h)\sqrt{1-h}\\
&+\overline{P}_{n-1}(h)\sqrt{1-h}\arcsin\sqrt{h}+\widetilde{\widetilde{P}}_{n-1}(h)\ln\frac{1+\sqrt{h}}{1-\sqrt{h}}\\
&+\overline{\overline{P}}_{n-1}(h)\ln(1-h)+P_{3n-3}(\sqrt{h})],\quad h\in(0,1),\\
\endaligned$$
where $P_i(u)$, $\widetilde{P}_i(u) $, $\widehat{P}_{i}(u) $, $\overline{P}_{i}(u) $, $ \widetilde{\widetilde{P}}_{i}(u)$, and $\overline{\overline{P}}_{i}(u)$ are polynomials of $u$ with degree $i$. Notice that
$$P_{3n-3}(\sqrt{h})=
\left\{\!\!\!
\begin{array}{ll}
P_{n-2+[\frac{n}{2}]}(h)\sqrt{h}+P_{n-2+[\frac{n}{2}]}(h),\quad &n=2k,\\ \\
P_{n-2+[\frac{n}{2}]}(h)\sqrt{h}+P_{n-1+[\frac{n}{2}]}(h),\quad &n=2k+1.\\
\end{array}
\right.
$$
Then, we have for $n\geq 3$,
$$\aligned
M(h)=&\frac{1}{(h-1)^{n-2}}[P_{m-1}(h)+P_{n-2+[\frac{n}{2}]}(h)\sqrt{h}+\widehat{P}_{n-1}(h)\sqrt{1-h}\\
&+\overline{P}_{n-1}(h)\sqrt{1-h}\arcsin\sqrt{h}+\widetilde{\widetilde{P}}_{n-1}(h)\ln\frac{1+\sqrt{h}}{1-\sqrt{h}}+\overline{\overline{P}}_{n-1}(h)\ln(1-h)],\\
\endaligned$$
where
\begin{equation}\label{m}
m=n+[\frac{n-1}{2}].
\end{equation}

Next, we will consider the $m$th order derivative of $M_1(h)$, where
$$M_1(h)=M(h)(h-1)^{n-2}.$$
For this purpose, we first state some facts.
$$\left(\arcsin \sqrt{h}\right)'=\frac{1}{2\sqrt{h}}\frac{1}{\sqrt{1-h}},$$
$$\left(\ln\frac{1+\sqrt{h}}{1-\sqrt{h}}\right)'=\frac{1}{1-h}\frac{1}{\sqrt{h}},$$
$$\left(\ln(1-h)\right)^{(n+1)}=-\frac{n!}{(1-h)^{n+1}},$$
$$\left(\arcsin\sqrt{h}\right)^{(n+1)}=\frac{P_n(h)}{[h(1-h)]^{n+\frac{1}{2}}},$$
$$\left(\sqrt{1-h}\arcsin \sqrt{h}\right)^{(n)}=\frac{P_{n-1}(h)}{(1-h)^{n-1}h^{n-\frac{1}{2}}}+\frac{-(2n-3)!!}{2^n(1-h)^{n-\frac{1}{2}}}\arcsin\sqrt{h},$$
and
$$\left(\ln\frac{1+\sqrt{h}}{1-\sqrt{h}}\right)^{(n+1)}=\frac{\overline{P}_{n}(h)}{(1-h)^{n+1}h^{n+\frac{1}{2}}}.$$
Then, the $m$th order derivative of $M_1(h)$ is as follows
\begin{equation}\label{y2M1}
\aligned
 \left(M_1(h)\right)^{(m)}&=\sum\limits_{i=0}^{n-2+[\frac{n}{2}]}C_m^iP_{n-2+[\frac{n}{2}]-i}(h)\left(\sqrt{h}\right)^{(m-i)}+\sum\limits_{i=0}^{n-1}
C_m^i\widehat{P}_{n-1-i}(h)\left(\sqrt{1-h}\right)^{(m-i)}\\
&+\sum\limits_{i=0}^{n-1}C_m^i\overline{P}_{n-1-i}(h)\left(\sqrt{1-h}\arcsin\sqrt{h}\right)^{(m-i)}
+\sum\limits_{i=0}^{n-1}C_m^i\widetilde{\widetilde{P}}_{n-1-i}(h)\left(\ln\frac{1+\sqrt{h}}{1-\sqrt{h}}\right)^{(m-i)}\\
&+\sum\limits_{i=0}^{n-1}C_m^i\overline{\overline{P}}_{n-1-i}(h)\left(\ln(1-h)\right)^{(m-i)}\\
&=\sum\limits_{i=0}^{n-2+[\frac{n}{2}]}\frac{(-1)^{m-i-1}C_m^i(2m-2i-3)!!P_{n-2+[\frac{n}{2}]-i}(h)}{2^{m-i}h^{m-i-\frac{1}{2}}}+\sum\limits_{i=0}^{n-1}
\frac{C_m^i(2m-2i-3)!!\widehat{P}_{n-1-i}(h)}{2^{m-i}(1-h)^{m-i-\frac{1}{2}}}\\
&+\sum\limits_{i=0}^{n-1}\frac{C_m^i\overline{P}_{n-1-i}(h)P_{m-1-i}(h)}{(1-h)^{m-i-1}h^{m-i-\frac{1}{2}}}+\sum\limits_{i=0}^{n-1}\frac{-(2m-2i-3)!!\overline{P}_{n-1-i}(h)}{2^{m-i}(1-h)^{m-i-\frac{1}{2}}}\arcsin{\sqrt{h}}\\
&+\sum\limits_{i=0}^{n-1}\frac{C_m^i \widetilde{\widetilde{P}}_{n-1-i}(h)P_{m-1-i}(h)}{(1-h)^{m-i}h^{m-i-\frac{1}{2}}}+\sum\limits_{i=0}^{n-1}\frac{-C_m^i(m-i-1)!\overline{\overline{P}}_{n-1-i}(h)}{(1-h)^{m-i}}\\
&=\frac{P_{n-2+[\frac{n}{2}]}(h)}{h^{m-\frac{1}{2}}}+\frac{\widehat{P}_{n-1}(h)}{(1-h)^{m-\frac{1}{2}}}+\frac{\overline{P}_{n-1}(h)P_{m-1}(h)}{(1-h)^{m-1}h^{m-\frac{1}{2}}}+\frac{\overline{P}^*_{n-1}(h)}{(1-h)^{m-\frac{1}{2}}}\arcsin\sqrt{h}\\
&+\frac{\widetilde{\widetilde{P}}_{n-1}(h)P_{m-1}(h)}{(1-h)^mh^{m-\frac{1}{2}}}+\frac{\overline{\overline{P}}_{n-1}(h)}{(1-h)^{m}}\\
&=\frac{P_{n+m-2}(h)+\frac{P_{n-2+[\frac{n}{2}]+m}(h)}{\sqrt{h}\sqrt{1-h}}+\frac{\widetilde{P}_{n+m-2}(h)}{\sqrt{1-h}}+\widehat{P}_{n-m+2}(h)\frac{\sqrt{1-h}}{\sqrt{h}}+\overline{P}_{n+m-2}(h)\arcsin\sqrt{h}}{(1-h)^{m-\frac{1}{2}}h^{m-1}}.
\endaligned
\end{equation}

Now, we consider the $(n+m-1)$th order derivative of $M_2(h)$, where
$$M_2(h)=\left(M_1(h)\right)^{(m)}(1-h)^{m-\frac{1}{2}}h^{m-1}.$$
 Notice that
$$\left(\frac{1}{\sqrt{h}\sqrt{1-h}}\right)^{(n)}=\frac{P_n(h)}{[h(1-h)]^{n+\frac{1}{2}}}$$
and
$$\left(\frac{\sqrt{1-h}}{\sqrt{h}}\right)^{(n)}=\frac{\overline{\overline{P}}_n(h)}{h^{n+\frac{1}{2}}(1-h)^{n-\frac{1}{2}}}.$$
Set $t=n+m-1$. Then, we have
$$
\aligned
 \left(M_2(h)\right)^{(t)}&=\sum\limits_{i=0}^{t}C_t^iP_{n+m-2+[\frac{n}{2}]-i}(h)\left(\frac{1}{\sqrt{1-h}\sqrt{h}}\right)^{(t-i)}+\sum\limits_{i=0}^{n+m-2}
C_t^i\widetilde{P}_{n+m-2-i}(h)\left(\frac{1}{\sqrt{1-h}}\right)^{(t-i)}\\
&+\sum\limits_{i=0}^{n+m-2}
C_t^i \widehat{P}_{n+m-2-i}(h)\left(\frac{\sqrt{1-h}}{\sqrt{h}}\right)^{(t-i)}+\sum\limits_{i=0}^{n+m-2}
C_t^i \overline{P}_{n+m-2-i}(h)\left(\arcsin\sqrt{h}\right)^{(t-i)}\\
&=\frac{\sum\limits_{i=0}^{t}C_t^iP_{n+m-2+[\frac{n}{2}]-i}(h)P_{t-i}(h)}{[h(1-h)]^{t-i+\frac{1}{2}}}+\frac{\sum\limits_{i=0}^{n+m-2}
-C_t^i(2t-2i-1)!!\widetilde{P}_{n+m-2-i}(h)}{2^{t-i}(1-h)^{t-i+\frac{1}{2}}} \\
&+\frac{\sum\limits_{i=0}^{n+m-2}
-C_t^i(2t-2i-1)!!\widehat{P}_{n+m-2-i}(h)\overline{\overline{P}}_{t-i}(h)}{h^{t-i+\frac{1}{2}}(1-h)^{t-i-\frac{1}{2}}}+\frac{\sum\limits_{i=0}^{n+m-2}
C_t^i\overline{P}_{n+m-2-i}(h)P_{t-1-i}(h)}{[h(1-h)]^{t-i-\frac{1}{2}}}\\
\endaligned
$$
$$
\aligned
&=\frac{P_{n+m-2+[\frac{n}{2}]}(h)P_t(h)}{[h(1-h)]^{t+\frac{1}{2}}}+\frac{\widetilde{P}_{n+m-2}(h)}{(1-h)^{t+\frac{1}{2}}}+\frac{\widehat{P}_{n+m-2}(h)\overline{\overline{P}}_t(h)}{h^{t+\frac{1}{2}}(1-h)^{t-\frac{1}{2}}}+\frac{\overline{P}_{n+m-2}(h)P_{t-1}(h)}{[h(1-h)]^{t-\frac{1}{2}}}\\
&=\frac{P_{n+m+t-2+[\frac{n}{2}]}(h)+P_{n+m+t-2}(h)\sqrt{h}}{[h(1-h)]^{t+\frac{1}{2}}}.
\endaligned
$$
Then, $\left(M_2(h)\right)^{(t)}=0$ is equivalent to
$$P_{n+m+t-2+[\frac{n}{2}]}(h)+P_{n+m+t-2}(h)\sqrt{h}=0. $$
Thus, the number of zeros of $\left(M_2(h)\right)^{(t)}$ is at most $ 2n+2m+2t-4+2[\frac{n}{2}]$. By Theorem \ref{upperbound1}, the number of zeros of $M_2(h)$ is at most $ 2n+2m+3t-4+2[\frac{n}{2}]$, multiplicity taken into account. By \eqref{y2M1} and Theorem \ref{upperbound1}, the number of zeros of $M_1(h)$ is at most $ 2n+3m+3t-4+2[\frac{n}{2}]$, multiplicity taken into account. Thus, by \eqref{m} and $t=n+m-1$, the number of zeros of $M(h)$ is at most $15n-13$ ($15n-11$) if $n\geq 3$ is even (odd), multiplicity taken into account. The first conclusion follows.

By \eqref{g4i} and \eqref{yruh2M2}, we have $M(h)=\frac{M_3(h)}{h-1}$ for $n=1,2,$ where
$$\aligned
M_3(h)&=P_2(h)+\widetilde{P}_2(h)\sqrt{h}+\widehat{P}_2(h)\sqrt{1-h}+\overline{P}_2(h)\sqrt{1-h}\arcsin\sqrt{h}\\
&+\overline{\overline{P}}_2(h)\ln\frac{1+\sqrt{h}}{1-\sqrt{h}}+\widetilde{\widetilde{P}}_2(h)\ln(1-h).\\
\endaligned$$
Then, the $3$th order derivative of $M_3(h)$ is as follows
\begin{equation}\label{yM3}
\aligned
\left(M_3(h)\right)^{(3)}&=\sum\limits_{i=0}^2C_3^i\widetilde{P}_{2-i}(h)\left(\sqrt{h}\right)^{(3-i)}+\sum\limits_{i=0}^2C_3^i\widehat{P}_{2-i}(h)\left(\sqrt{1-h}\right)^{(3-i)}\\
&+\sum\limits_{i=0}^2C_3^i\overline{\overline{P}}_{2-i}(h)\left(\ln\frac{1+\sqrt{h}}{1-\sqrt{h}}\right)^{(3-i)}+\sum\limits_{i=0}^2C_3^i\overline{P}_{2-i}(h)\left(\sqrt{1-h}\arcsin \sqrt{h}\right)^{(3-i)}\\
&+\sum\limits_{i=0}^2C_3^i\widetilde{\widetilde{P}}_{2-i}(h)\left(\ln(1-h)\right)^{(3-i)}\\
\endaligned
\end{equation}
$$
\aligned
&=\frac{\sum\limits_{i=0}^2(-1)^{2-i}C_3^i(3-2i)!!\widetilde{P}_{2-i}(h)}{2^{3-i}h^{\frac{5}{2}-i}}+\frac{\sum\limits_{i=0}^2(-1)C_3^i(3-2i)!!\widehat{P}_{2-i}(h)}{2^{3-i}(1-h)^{\frac{5}{2}-i}}\\
&+\frac{\sum\limits_{i=0}^2C_3^i\overline{\overline{P}}_{2-i}(h)P_{2-i}(h)}{(1-h)^{3-i}h^{\frac{5}{2}-i}}
+\frac{\sum\limits_{i=0}^2C_3^i\overline{P}_{2-i}(h)P_{2-i}(h)}{(1-h)^{2-i}h^{\frac{5}{2}-i}}\\
&+\frac{\sum\limits_{i=0}^2(-1)C_3^i(3-2i)!!\overline{P}_{2-i}(h)}{2^{3-i}(1-h)^{\frac{5}{2}-i}}\arcsin\sqrt{h}
+\frac{\sum\limits_{i=0}^2(-1)(2-i)!C_3^i\widetilde{\widetilde{P}}_{2-i}(h) }{(1-h)^{3-i}}\\
&=\frac{\widetilde{P}_2(h)}{h^{\frac{5}{2}}}+\frac{\widehat{P}_2(h)}{(1-h)^{\frac{5}{2}}}+\frac{\overline{P}_2(h)P_2(h)}{(1-h)^2h^{\frac{5}{2}}}+\frac{\overline{P}^*_2(h)}{(1-h)^{\frac{5}{2}}}\arcsin\sqrt{h}+\frac{\overline{\overline{P}}_2(h)P_2(h)}{(1-h)^3h^{\frac{5}{2}}}+\frac{\widetilde{\widetilde{P}}_2(h)}{(1-h)^3}\\
&=\frac{P_4(h)+\widehat{P}_4(h)\frac{\sqrt{1-h}}{\sqrt{h}}+\frac{\overline{P}_4(h)}{\sqrt{h}\sqrt{1-h}}+\frac{\overline{\overline{P}}_4(h)}{\sqrt{1-h}}+\widetilde{\widetilde{P}}_4(h)\arcsin\sqrt{h}}{h^2(1-h)^{\frac{5}{2}}}.\\
\endaligned
$$
Let
$$M_4(h)=P_4(h)+\widehat{P}_4(h)\frac{\sqrt{1-h}}{\sqrt{h}}+\frac{\overline{P}_4(h)}{\sqrt{h}\sqrt{1-h}}+\frac{\overline{\overline{P}}_4(h)}{\sqrt{1-h}}+\widetilde{\widetilde{P}}_4(h)\arcsin\sqrt{h}.$$
Notice that
$$\left(\frac{\sqrt{1-h}}{\sqrt{h}}\right)^{(n)}=\frac{P_n(h)}{h^{n+\frac{1}{2}}(1-h)^{n-\frac{1}{2}}}.$$
Then, we have
$$
\aligned
\left(M_4(h)\right)^{(5)}&=\sum\limits_{i=0}^4C_5^i\widehat{P}_{4-i}(h)\left(\frac{\sqrt{1-h}}{\sqrt{h}}\right)^{(5-i)}+\sum\limits_{i=0}^4C_5^i\overline{P}_{4-i}(h)\left(\frac{1}{\sqrt{1-h}\sqrt{h}}\right)^{(5-i)}\\
&+\sum\limits_{i=0}^4C_5^i\overline{\overline{P}}_{4-i}(h)\left(\frac{1}{\sqrt{1-h}}\right)^{(5-i)}+\sum\limits_{i=0}^4C_5^i\widetilde{\widetilde{P}}_{4-i}(h)\left(\arcsin\sqrt{h}\right)^{(5-i)}\\
&=\frac{\sum\limits_{i=0}^4C_5^i\widehat{P}_{4-i}(h)P_{5-i}(h)}{h^{\frac{11}{2}-i}(1-h)^{\frac{9}{2}-i}}+\frac{\sum\limits_{i=0}^4C_5^i\overline{P}_{4-i}(h)P^*_{5-i}(h)}{h^{\frac{11}{2}-i}(1-h)^{\frac{11}{2}-i}}\\
&+\frac{\sum\limits_{i=0}^4C_5^i(10-2i-1)!!\overline{\overline{P}}_{4-i}(h)}{2^{5-i}(1-h)^{\frac{11}{2}-i}}+\frac{\sum\limits_{i=0}^4C_5^i\widetilde{\widetilde{P}}_{4-i}(h)P_{4-i}(h)}{h^{\frac{9}{2}-i}(1-h)^{\frac{9}{2}-i}}\\
&=\frac{\widehat{P}_4(h)P_5(h)}{h^{\frac{11}{2}}(1-h)^{\frac{9}{2}}}+\frac{\overline{P}_4(h)P^*_5(h)}{h^{\frac{11}{2}}(1-h)^{\frac{11}{2}}}+\frac{\overline{\overline{P}}_4(h)}{(1-h)^{\frac{11}{2}}}+\frac{\widetilde{\widetilde{P}}_4(h)P_4(h)}{h^{\frac{9}{2}}(1-h)^{\frac{9}{2}}}\\
&=\frac{P_{10}(h)+P_9(h)\sqrt{h}}{h^{\frac{11}{2}}(1-h)^{\frac{11}{2}}}.
\endaligned
$$
Obviously, $\left(M_4(h)\right)^{(5)}=0$ is equivalent to
$$P_{10}(h)+P_9(h)\sqrt{h}=0. $$
Then, $\left(M_4(h)\right)^{(5)}$ has at most $20$ zeros. By Theorem \ref{upperbound1}, $M_4(h)$ has at most $25$ zeros, multiplicity taken into account. Thus, by \eqref{yM3} and Theorem \ref{upperbound1}, $M_3(h)$ has at most $28$ zeros, multiplicity taken into account. It implies that $M(h)$ has at most $28$ zeros for $n=1,2$, multiplicity taken into account.

From the above analysis and Theorem 4.4 in \cite{Han2021}, we can conclude that the number of limit cycles of system \eqref{yruh2} bifurcating from the period annulus $ \cup_{h\in(0,1)}\widehat{L}_h$ around the isochronous centers $(1,0)$ is no more than $ 15n-11$ (counting multiplicity) for $n\geq3$; $28$ (counting multiplicity) for $n=1,2.$ The proof is ended.
\end{proof}
\begin{remark}
Theorems \ref{yfirst} and \ref{ysecond} improve conclusions $(i)$ and $(ii)$ of Theorem \ref{301}, respectively.
\end{remark}
\section{Conclusion}\label{section5}
According to the proofs of Theorems \ref{GHIR1}, \ref{yfirst}, and \ref{ysecond}, it is easy to see that we can get rid of logarithm function, arc sine function and arc tangent function by Theorem \ref{upperbound1}. If the first order Melnikov function $M(h)$ is a linear combination of power functions and some other elementary functions, such as logarithm function or inverses of trigonometric functions, one can try to use Theorem \ref{upperbound1} to get rid of these elementary functions to find an upper bound of the number of zeros of $M(h)$. In fact, the authors in \cite{Cen2018,Chen2020,Xiong2012} used this idea to get rid of some logarithm functions to estimate upper bound of the number of zeros of the first order Melnikov function.

%%%%%%%%%%%%%%%%%%%%%%%%%%%%%%%%%%%%%%%%%%%%%%%%%%%%%%%%%%%%%%%%%%%%%%

%%%%%%%%%%%%%%%%%%%%%%%%%%%%%%%%%%%%%%%%%%%%%%%%%%%%%%%
%%% Acknowledgements. 致谢
%%%%%%%%%%%%%%%%%%%%%%%%%%%%%%%%%%%%%%%%%%%%%%%%%%%%%%%
\Acknowledgements{This work was supported by National Natural Science Foundation of China (Grant No.11931016 and 11771296) and Hunan Provincial Education Department (Grant No.19C1898).}

%%%%%%%%%%%%%%%%%%%%%%%%%%%%%%%%%%%%%%%%%%%%%%%%%%%%%%%
%%% Conflict of interest. 作者利益声明
%%%%%%%%%%%%%%%%%%%%%%%%%%%%%%%%%%%%%%%%%%%%%%%%%%%%%%%
%\InterestConflict

%%%%%%%%%%%%%%%%%%%%%%%%%%%%%%%%%%%%%%%%%%%%%%%%%%%%%%%
%%% Supplements. 补充材料, 非必选
%%%%%%%%%%%%%%%%%%%%%%%%%%%%%%%%%%%%%%%%%%%%%%%%%%%%%%%
%\Supplements{}

%%%%%%%%%%%%%%%%%%%%%%%%%%%%%%%%%%%%%%%%%%%%%%%%%%%%%%%
%%% Reference section. 参考文献
%%% citation in the content using "some words~\cite{1,2}".
%%% ~ is needed to make the reference number is on the same line with the word before it.
%%%%%%%%%%%%%%%%%%%%%%%%%%%%%%%%%%%%%%%%%%%%%%%%%%%%%%%

%%%%%%%%%%%%%%%%%%%%%%%%%%%%%%%%%%%%%%%%%%%%%%%%%%%%%%%
%%% Appendix sections. 附录章节, 非必选
%%%%%%%%%%%%%%%%%%%%%%%%%%%%%%%%%%%%%%%%%%%%%%%%%%%%%%%

\end{document}